\documentclass[11pt, a4paper]{amsart}
\usepackage{amsmath}
\usepackage{amssymb}
\usepackage{color}

\newtheorem{prop}{Proposition}[section]
\newtheorem{teo}{Theorem}[section]
\newtheorem{lema}{Lemma}[section]

\newtheorem{defi}{Definition}[section]
\newtheorem{rem}{Remark}[section]

\newcommand{\usup}{\overline{u}}
\newcommand{\bsub}{\underline{b}}
\newcommand{\bsup}{\overline{b}}
\def\ds{\displaystyle}

\def\R{\mathbb R}

\begin{document}

\title[Stable solutions of equations with a quadratic gradient term]{Stable solutions of equations with a quadratic gradient term}

\author[J. Terra ]{Joana Terra}


\address{Joana Terra\hfill\break\indent
Departamento  de Matem{\'a}tica, FCEyN \hfill\break\indent UBA (1428)
Buenos Aires, Argentina.} \email{{\tt jterra@dm.uba.ar} }

\date{}

\begin{abstract}
We study existence and regularity properties of stable positive solutions to the nonvariational problem $-\Delta u-b(x)|\nabla u|^2=\lambda g(u)$ in a bounded smooth domain. In the case where $b$ is constant, by means of a Hopf-Cole transformation, the problem can be taken to a variational form, for which there are classical results of Crandall-Rabinowitz, Mignot-Puel and Brezis-V\'azquez. In this paper we obtain results for a general bounded function $b=b(x)$ which coincide with the classical ones in the constant $b$ case.
\end{abstract}

\maketitle

\date{}

\begin{section}{Introduction}
\label{Intro} \setcounter{equation}{0}

%\begin{section}{Stable solutions of equations with a quadratic gradient term}\label{quad}

In this paper we are interested in the existence and qualitative
properties of positive solutions to equations
of the form 
\[
-\Delta u -b(x)|\nabla u|^2= \lambda g(u)
\]
in a bounded 
smooth domain $\Omega$ of $\R^n$, for $\lambda\geq 0$, $b=b(x)$ a given function and 
$g$ an increasing nonlinearity with $g(0)>0$. This type of
equations arise in different contexts from physics to stochastic
processes. Equations with the
quadratic gradient term 
$-\Delta -b(x)|\nabla u|^2$ appear in relation to different  contexts within the
literature. If $b=b(x)$ is constant and positive, the equation can be
thought as the stationary part of the parabolic equation
$u_t-\epsilon\Delta u=|\nabla u|^2$ which in turn may be seen as the
viscosity approximation, as $\epsilon$ tends to $0^+$, of
Hamilton-Jacobi equations from stochastic control theory \cite{L}. In
\cite{KPZ} the same equation (known in this context as Kardar-Parisi-Zhang equation) arises related to the physical theory of
growth and roughening on surfaces.  Also classical are the existence
results for equations involving a quadratic gradient term and such
that $b=b(u)$ (see for instance \cite{LU, LL}). For more on such
equations with $b=b(u)$ see for example \cite{ADP}.
\vspace{1em}

If the case the coefficient function $b$ is constant, the above equation can
be transformed, using the so called Hopf-Cole transformation, into the
equation  $-\Delta v = \lambda f(v)$, where $f$ satisfies
the same hypothesis as $g$. This simpler equation for $v$ appears in
many different contexts and has been extensively studied. This family
of equations includes, for example, the Gelfand problem, where
$f(v)=e^v$ with zero Dirichlet boundary conditions on the boundary of
$\Omega=B_1$, the unit ball.
Some first results concerning this problem involved the construction of 
 explicit radial solutions in dimensions $2$ and $3$, and in the special
case where $\lambda=2$ and $n=3$ it was established that there are infinitely many solutions.

The natural question that arises regarding the equation for $v$ is the
study of the solutions $(\lambda, v)$, their existence and 
properties. The classical existence result says there exists a finite extremal parameter
$\lambda^*$ such that for $\lambda>\lambda ^*$ there exist no bounded
solutions $v$, whereas for $0<\lambda<\lambda^*$ there exists a minimal
(i.e., smallest) bounded solution $v_{\lambda}$. Moreover, the branch
$\{v_{\lambda}\}$ is increasing in $\lambda$ and each solution
$v_{\lambda}$ is stable. A more
delicate problem is the study of the increasing limit $v^*=\lim_{\lambda\uparrow\lambda^*}
v_{\lambda}$, which turns out to be a weak solution of the problem with
parameter $\lambda^*$. However $v^*$ may be either bounded or
singular, depending on the domain $\Omega$ and the nonlinearity $f$.

In the case where $\Omega$
is the unit ball of $\R^n$ and $f(v)=e^v$, Joseph and Lundgren
\cite{JL} completely described the
existence and regularity of solutions in terms of $\lambda$. Their
result also applies to the other classical model, that is, when
$f(v)=(1+v)^p$ and $p>1$. For general domains Crandall and Rabinowitz
\cite{CR} and Mignot and Puel \cite{MP} gave sufficient conditions for
the extremal solution $v^*$ to be classical, when the nonlinearity $f$
is either exponential or power like. Brezis and V\'azquez \cite{BV}
raised the question of studying when is the extremal solution bounded for general
$f$ convex, depending on the dimension $n$ and the domain $\Omega$. For $n\leq 3$ Nedev \cite{Ne} proved that $v^*$ is bounded
for any domain $\Omega$. More recently in \cite{C} Cabr\'e proves that
the extremal solution is bounded if the domain
$\Omega$ is convex and $n\leq 4$. For higher dimensions the
only known result so far for general $f$ concerns radial solutions. Namely, Cabr\'e
and Capella \cite{CC2} prove that if $\Omega$ is the unit ball and
$n\leq 9$ then $v^*$ is bounded for every $f$.

In this paper we derive similar results for the case where $b(x)$ is
non-constant, and hence the problem is not variational. Although there
is no energy functional associated to our problem, and hence there is no
quadratic form, we are still interested in ``stable'' solutions. To
define stability of a solution to a non-variational problem we will use a
different condition than the one used in the variational setting (see section 2).

\vspace{1em}

In section 3, for some special nonlinearities $g$, we
derive the stability of the classical solutions. 
In addition, for the class of stable solutions, we  prove new
regularity results
 involving conditions on the function $b(x)$ and the dimension $n$. 

\vspace{1em}

In the following section we establish an existence theorem in terms
of $\lambda$. The result is similar to
the one in the classical context with $b\equiv 0$. Namely we prove the existence of an
extremal parameter $\lambda^*$ such that for $\lambda>\lambda^*$ there
is no solution, whereas for $0<\lambda<\lambda^*$ there is a minimal
classical solution $u_{\lambda}$. Moreover, for $g(u)=e^u$ and some dimensions $n$
we are able to prove that the minimal classical solutions
$u_{\lambda}$ are stable and that the extremal function
$u^*=\lim_{\lambda\rightarrow\lambda^*}u_{\lambda}$ is a weak solution
for $\lambda=\lambda^*$.  As before, in the case where $b(x)=b$ is constant, the existence
  result coincides with the classical one.

%{ \bf Theorem \ref{prop1}}
%{\it Let $b=b(x)\geq 0$ be a $C^{\alpha}(\overline{\Omega})$ function defined in a smooth bounded domain $\Omega\subset\R^n$  and
%$g$ be a nondecreasing $C^1$ function with $g(0)>0$ and
%$\lim_{u\rightarrow+\infty}\frac{g(u)}{u}=+\infty$. Then, there
%exists a parameter $0<\lambda^{*}<\infty $ such that:
%
%
%\noindent {\rm (a)} If $\lambda > \lambda^{*}$ then there is no
%classical solution of
%$(\ref{lambdau})$.\\
%{\rm (b)} If $0\leq \lambda < \lambda^{*}$ then there exists a minimal
%classical solution $u_\lambda$ of $(\ref{lambdau})$.
%Moreover, $u_\lambda<u_\mu$ if $\lambda<\mu<\lambda^*$. 
%
%In addition, if $g(u)=e^u$ and for every positive constants $\delta$
%and $\eta$ with $\delta^2+\eta^2\leq 1$, the function $b$ satisfies $0\leq\bsub\leq b
%\leq\bsup$ in $\Omega$ for constants $\bsub$ and $\bsup$ such that 
%$(\bsup-\bsub)<\frac{\delta^2}{\eta^2}(\eta^2-\frac{\bsup}{8})$ and 
%$n <2 + 4\eta^2+4\bsup
%+4\sqrt{\eta^2(\eta^2+\bsup)-2\bsup(\bsup-\bsub)\frac{\eta^2}{\delta^2}}$,
%then $u_\lambda$ is semi-stable. Moreover, the limit
%$u^*=\lim_{\lambda\rightarrow\lambda^*}u_{\lambda}$ is a weak solution
%of ($\ref{lambdau}$) for $\lambda=\lambda^*$. That is, it satisfies
%$$-\int_{\Omega}u^*\Delta\xi -\int_{\Omega}b(x)|\nabla u^*|^2 \xi= 
%\lambda^*\int_{\Omega}e^{u^*}\xi,$$
%for every $\xi\in C^2(\overline{\Omega})$ with $\xi= 0$ on
%$\partial\Omega$. In addition, the estimates of Proposition
%$\ref{propreg}$ apply to $u^*$.}

%\begin{rem}

%\end{rem}

\vspace{1em}

Finally in the last section we establish some sufficient conditions for stable solutions $u$
to be in $H^1(\Omega)$. Once again we
consider two cases separately: $b$ positive and $b$ negative.
On the one hand, we have the case where
$b(x)=b>0$ is constant and positive. In this setting we are able to
prove an $H^1(\Omega)$ result following the similar technique of the
classical case (see Brezis-V\'azquez \cite{BV}) which requires an extra condition on
$g$. We note here that via the Hopf-Cole transformation, one could use
the classical result to obtain a condition for $e^u$ to be in
$H^1(\Omega)$. This would imply, of course, that $u$ is also in
$H^1(\Omega)$ but this gives a 
much stronger condition on $g$ than the more optimal one that we prove. On the other hand, using different techniques, namely truncations as introduced by
Boccardo \cite{B}, we prove the  $H^1(\Omega)$ result for every solution (not
necessarily stable) 
with $b(x)$ strictly negative, and any $L^1$ nonlinearity $g$.

%\vspace{1em}
%
%{\bf Proposition \ref{teoH1b}}
%{\it Let $b>0$ be a constant and $u$ a positive classical solution of the problem
%$-\Delta u -b|\nabla u|^2 = \lambda g(u)$ with zero Dirichlet boundary
%conditions, $g\geq 0$ and $g'>0$ in $\Omega$ and $\lambda>0$ a parameter. 
%Assume that $u$ is a stable solution
%Then, if 
%$$\liminf_{s\rightarrow\infty}\frac{g'(s)(e^{bs}-1)}{bg(s)}>1,$$
%we have that $||u\||_{H^1(\Omega)}\leq C$ where $C$ is independent of $\lambda$.}
%

%\vspace{1em}

%When $b$ tends to zero, this condition on $g$ formally gives
%$$\liminf_{s\rightarrow\infty}\frac{g'(s)s}{g(s)}>1,$$
%which is the condition given by Brezis-V\'azquez in the classical context.

%\vspace{1em}
%
%{\bf Proposition \ref{teoH1b-}}
%{\it Let $b(x)\leq -\epsilon<0$ for some $\epsilon>0$ and $u$ a positive classical
%solution to the problem $-\Delta u - b(x)|\nabla u|^2=\lambda g(u)$
%with zero Dirichlet boundary conditions, $\lambda>0$ a parameter, and
%assume that $g(u)\in L^1(\Omega)$.
%Then, $||u||_{H^1(\Omega)}\leq C$, where $C$ is independent of $\lambda$.}

\end{section}

\begin{section}{Preliminaries}
We are interested in  nonnegative solutions of the problem
\begin{equation}\label{eqg}
\left\{\begin{array}{rcll}
       -\Delta u -b(x)|\nabla u|^2 & =& \lambda g(u) & {\rm in} \; \Omega\\
               u & = & 0 & {\rm on} \;\partial\Omega,
       \end{array}
\right.
\end{equation}
where $\Omega\subset\R^n$ is a bounded smooth domain,
$g:[0,+\infty)\rightarrow \R$ is a given nonlinearity, $\lambda\geq0$ is
  a parameter and $b=b(x)$ is a
  bounded function. The properties of
  the solutions $u$ of ($\ref{eqg}$) will depend on the coefficient function
  $b$ and hence we will distinguish different cases.

We are interested in the class of stable solutions of \eqref{eqg}. The usual
definition of stability requires the equation to be of variational type, since stability is determined by the sign of the second variation of the energy associated to the variational equation. In order to define
stability for a wider family of problems we use the linearized
equation instead.

\begin{defi}\label{stable5}
Let $u$ be a classical solution of  problem \eqref{eqg}. We say that $u$ is {\it
  stable} if there exists a function
  $\phi\in W^{2,p}(\Omega)$ for some $p>n$ such that $\phi>0$ in
  $\overline{\Omega}$  and
\begin{equation}\label{phistable}
-\Delta \phi -2b(x)\nabla u\nabla\phi \geq \lambda g'(u)\phi \quad {\rm
 in}\;\Omega.
\end{equation}
\end{defi}
Note that in the variational setting (for example \eqref{eqg} with
$b\equiv 0$), the existence of such
a  supersolution $\phi$, positive in $\overline{\Omega}$, is equivalent to saying that $u$ is stable
in the variational sense, i.e., the quadratic form
$Q_u(\xi)$ defined there is positive for every test function
$\xi\not\equiv 0$. Equivalently,
 the first eigenvalue of the linearized problem is positive (see \cite{BNV}). 

\vspace{1em}

However, for a general
function $b$, problem
($\ref{eqg}$) is not self-adjoint and therefore we bypass this
difficulty by considering the existence of $\phi$ instead of working
with the quadratic form $Q$ used before, which makes no sense (or does
not exist) for non self-adjoint problems. 
\end{section}

\begin{section}{Case $b(x)\equiv b$ is constant}

In this section we consider the case where the coefficient function
$b$ is constant, that is, we study nonnegative solutions to

\begin{equation}\label{equb}
\left\{
\begin{array}{rcll}
-\Delta u-b|\nabla u|^2 & = & \lambda g(u)& \text{ in }\Omega\\
u & = & 0 & \text{ on }\partial\Omega.
\end{array}\right.
\end{equation}

This problem can be easily transformed
into a classical semilinear elliptic equation for a new
function $v$ and a new nonlinearity $f=f(v)$ that depends on
$g$. Since the transformation, called Hopf-Cole transformation,
 depends on the sign of the constant
$b$, we will treat both cases separately in the next two
subsections. The nonlinearities $f$  that arise in these two cases
are quite different. Nevertheless, if $g(u)=e^{\beta u}$ for some
constant $\beta$, the
classical regularity results for $v$ and our regularity results for
$u$ 
(that we later generalize
to $b=b(x)$) agree regardless of the sign of
$b$.
 
\begin{subsection}{Case $b=ctt>0$}

Let $b$ be constant and positive, that is, $b(x)\equiv b>0$. 
In this special case we can perform the
Hopf-Cole transformation $v=e^{bu}-1$. The new nonnegative function $v$
satisfies

\begin{equation}\label{eqv}
\left\{\begin{array}{rcll}
       -\Delta v &= &\lambda b(v+1)g\left(\frac{1}{b}\log(v+1)\right) & {\rm in} \; \Omega\\
               v & = & 0 & {\rm on} \;\partial\Omega.
       \end{array}
\right.
\end{equation}

We will denote by $f$ the nonlinearity appearing on the right-hand
side of the equation above, that is, $v$ satisfies
\begin{equation}\label{expf}
-\Delta v =\lambda f(v) \quad {\rm in} \;\Omega, \text{ where } \,
f(v)=b(v+1)g\left(\frac{1}{b}\log(v+1)\right).
\end{equation}

A first example is the one we obtain letting $g(u)=e^{e^{b u}-1-bu}$ ,
i.e., considering the equation
$$-\Delta u -b|\nabla u|^2 = \lambda e^{e^{b u}-1-bu}.$$
For this choice of $g$ the equation for $v$ becomes
$-\Delta v = \lambda be^v,$ the classical exponential
nonlinearity. We know (see \cite{CR}) that every stable weak
solution satisfies $v\in L^{\infty}(\Omega)$ if $n\leq 9$.

\vspace{1em}

Another example is the one we obtain letting $g(u)=e^{\beta u}$ for
some constant $\beta>0$, i.e.,
\begin{equation}\label{522}
\left\{\begin{array}{rcll}
-\Delta u -b|\nabla u|^2& =& \lambda e^{\beta u}& \text{ in } \Omega\\
u & = & 0 & \text{ on }\partial\Omega.
\end{array}\right.
\end{equation} 
Then, the equation for $v$ becomes
$ -\Delta v = \lambda b(v+1)^p,$
where $p=1+\frac{\beta}{b}>1$. This is the classical power nonlinearity
case for $v$. For this equation it is known (see \cite{BV}) that, if
$v$ is a $H_0^1(\Omega)$ semi-stable solution then
\[
v\in L^{\infty}(\Omega) \quad \textrm{ if } \quad 
\left\{\begin{array}{lcl}
n\leq 10 & \text{or}&  \\ 
 n>10 & \text{and} & {\ds p< \frac{n-2\sqrt{n-1}}{n-4-2\sqrt{n-1}}}, 
\end{array}\right.
\]
 that is, if $n\leq 10$ or $10<n<
2+\frac{4p}{p-1} + 4\sqrt{\frac{p}{p-1}}$. In our case, for
$p=1+\beta/b$ we have that
\begin{equation}\label{vbctt}
v\in L^{\infty}(\Omega) \quad \textrm{ if } \quad n\leq 10\; \text{ or
}\;10<n<
6+ 4\frac{b}{\beta}+4\sqrt{1+\frac{b}{\beta}}.
\end{equation}

Note that our stability assumption on $u$, that is,
the existence of a function $\phi$, positive in $\overline{\Omega}$,
satisfying \eqref{phistable}
is equivalent to the existence of a function $\psi= e^{b u}\phi$,
positive in $\overline{\Omega}$, 
satisfying
$$-\Delta \psi \geq \lambda f'(v)\psi,$$ where $v=e^{bu}-1$ and $f$ is
given by \eqref{expf}, which is
in turn equivalent to the stability of $v$.

Now, since
$v=e^{bu}-1$  we may conclude that, for every stable classical solution $u$ of \eqref{522},
$$u\in L^{\infty}(\Omega) \quad \textrm{ if } \quad n\leq 10\; \text{ or
  }\; 10<n<
6+4\frac{b}{\beta}+4\sqrt{1+\frac{b}{\beta}},$$
and that this is a uniform $L^{\infty}$ estimate for all stable
  solutions (as the one for $v$ in \eqref{vbctt}). In particular this
  establishes a uniform bound for all minimal solutions $u_{\lambda}$
  and therefore yields a sufficient condition for the extremal weak solution $u^*$
  to be in $L^{\infty}(\Omega)$.
That is, we have the following
\begin{prop}\label{propb+}
Let $b>0$ and $u$ a positive classical stable solution to 
\begin{equation*}
\left\{\begin{array}{rcll}
-\Delta u -b|\nabla u|^2& = & \lambda e^{\beta u}&\text{ in } \Omega\\
u & = & 0&\text{ on }\partial\Omega,
\end{array}\right.
\end{equation*} where  $\lambda>0$ is a parameter. Then
$$ ||u||_{L^{\infty}(\Omega)}\leq C \quad \textrm{ if } \quad n\leq 10\; \text{ or
  }\; 10<n<
6+4\frac{b}{\beta}+4\sqrt{1+\frac{b}{\beta}},$$ where $C$ is a
  constant depending only on $n, b, \beta$ and $\Omega$ (in particular
  is independent of $\lambda$).
\end{prop}

Let us now prove directly this result, in the case where $\beta>b/8$, by using the equation for $u$ and
the fact that we are assuming $u$ stable. As we will
see, for such $\beta$, we reach the same optimal result. The motivation for the
following calculations is that in the case where $b(x)$ is
non-constant, we are forced to work with the equation for $u$, since
there is, in principle, no transformation to a classical semilinear problem
without terms involving the square of the gradient. 

We begin by establishing a technical lemma.

\begin{lema}\label{lemmaprop}
Let $b>0$ and $u$ be a positive classical solution to 
\begin{equation*}
\left\{\begin{array}{rcll}
-\Delta u -b|\nabla u|^2& = & \lambda e^{\beta u}&\text{ in } \Omega\\
u & = & 0&\text{ on }\partial\Omega,
\end{array}\right.
\end{equation*} where $\lambda>0$ is a parameter and
$\beta>0$. For $\gamma\in {\mathbb N}$ satisfying $\gamma\geq 2$, we have
\begin{equation}\label{intlemma}
\int_{\Omega}|\nabla u|^2e^{\gamma
  bu}(e^{bu}-1)^{2\alpha-2}dx\leq\frac{\lambda}{b(2\alpha+\gamma-3)}\int_{\Omega}
  e^{(\beta+(2\alpha+\gamma-2)b)u}dx+\lambda L_{\gamma},
\end{equation} 
where $\alpha>1/2$ is a parameter and $L_{\gamma}$ is a linear combination of the
$\gamma-2$ integrals $\int_{\Omega}e^{(\beta+(2\alpha+k)b)u}dx$,
$k=0,1,\dots, \gamma-3$, with coefficients depending only on
$b,\alpha$ and $\gamma$.
\end{lema}

\begin{proof}
Let $2\alpha>1$ and $\gamma\geq 2$ an integer. We have,
\begin{eqnarray*}
 \hspace{-1em}{\ds \int_{\Omega}|\nabla u|^2e^{\gamma bu}(e^{bu}-1)^{2\alpha-2}} & \hspace{-1em} = 
& \hspace{-1em}{\ds  \int_{\Omega} \nabla u e^{(\gamma-1)bu}\nabla ue^{bu}(e^{bu}-1)^{2\alpha-2}}\\
& \hspace{-1em} = & \hspace{-1em} {\ds  \int_{\Omega} \nabla u e^{(\gamma-1)bu}\frac{\nabla (e^{bu}-1)^{2\alpha-1}}{b(2\alpha-1)}}\\
& \hspace{-1em} = &  \hspace{-1em} {\ds\int_{\Omega} \frac{
  e^{(\gamma-1)bu}(e^{bu}-1)^{2\alpha-1}}{b(2\alpha-1)}\left(-\Delta
u-(\gamma-1)b|\nabla u|^2\right).}
\end{eqnarray*}
Using the equation for $u$ we have
\begin{eqnarray*}
{\ds \int_{\Omega}|\nabla u|^2e^{\gamma bu}(e^{bu}-1)^{2\alpha-2}} & = 
& {\ds \int_{\Omega} \frac{
  e^{(\gamma-1)bu}(e^{bu}-1)^{2\alpha-1}}{b(2\alpha-1)}\left(\lambda e^{\beta u}
-(\gamma-2)b|\nabla u|^2\right)}\\
& \leq & {\ds  \frac{\lambda}{b(2\alpha-1)}\int_{\Omega}
e^{(\beta+(2\alpha+\gamma-2)b)u}-}\\
& &{\ds \hspace{-12em}-\frac{\gamma-2}{2\alpha-1}\int_{\Omega}|\nabla u|^2e^{\gamma
  bu}(e^{bu}-1)^{2\alpha-2}+ \frac{\gamma-2}{2\alpha-1}\int_{\Omega}|\nabla u|^2e^{(\gamma-1)bu}(e^{bu}-1)^{2\alpha-2}.}
 \end{eqnarray*}
This yields, adding the left hand side to the second term on the right
hand side, and since $2\alpha>1$ and $\gamma\geq 2$ 
(and thus $2\alpha+\gamma-3>0$),
\begin{eqnarray}\label{gamma}
\int_{\Omega}|\nabla u|^2e^{\gamma bu}(e^{bu}-1)^{2\alpha-2}& \leq &
 \frac{\lambda}{b(2\alpha+\gamma-3)}\int_{\Omega}
e^{(\beta+(2\alpha+\gamma-2)b)u}+\\
\nonumber& & + \frac{\gamma-2}{2\alpha+\gamma-3}\int_{\Omega}|\nabla u|^2e^{(\gamma-1)bu}(e^{bu}-1)^{2\alpha-2}.
\end{eqnarray}
If $\gamma=2$ the second term on the right hand side of \eqref{gamma}
is zero and we conclude \eqref{intlemma} (as desired) with
$L_{\gamma}=0$.
 Otherwise, for $\gamma\in{\mathbb N}$, $\gamma\geq 3$, we may repeat the computations
above with $\gamma$ replaced by $\gamma-1$ to obtain
\begin{eqnarray*}
\int_{\Omega}|\nabla u|^2e^{(\gamma-1)bu}(e^{bu}-1)^{2\alpha-2}& \leq &
 \frac{\lambda}{b(2\alpha+\gamma-4)}\int_{\Omega}
e^{(\beta+(2\alpha+\gamma-3)b)u}+\\
& & + \frac{\gamma-3}{2\alpha+\gamma-4}\int_{\Omega}|\nabla u|^2e^{(\gamma-2)bu}(e^{bu}-1)^{2\alpha-2}.
\end{eqnarray*}
Note that the left hand side is the second integral on the right hand
side of \eqref{gamma}, which remains to be controlled. Note also that the exponent of the exponential function in the first
integral on the right hand side decreases on each iteration. We may
continue this process until we are left with the integral
 of $|\nabla u|^2e^{2bu}(e^{bu}-1)^{2\alpha-2}$. For this integral,
 $$\int_{\Omega} |\nabla
 u|^2e^{2bu}(e^{bu}-1)^{2\alpha-2}\leq\frac{\lambda}{b(2\alpha-1)}\int_{\Omega}
 e^{(\beta+2\alpha b)u}.$$
Hence,
\begin{equation*}
\int_{\Omega}|\nabla u|^2e^{\gamma
  bu}(e^{bu}-1)^{2\alpha-2}dx\leq\frac{\lambda}{b(2\alpha+\gamma-3)}\int_{\Omega}
  e^{(\beta+(2\alpha+\gamma-2)b)u}dx+\lambda L_{\gamma},
\end{equation*} 
where $L_{\gamma}$ is a linear combination of the
$\gamma-2$ integrals $\int_{\Omega}e^{(\beta+(2\alpha+k)b)u}dx$,
$k=0,1,\dots, \gamma-3$, with coefficients depending only on
$b,\alpha$ and $\gamma$.
\end{proof}

Next we use the assumption that $u$ is stable according
to Definition $\ref{stable5}$, that is, there exists a positive
function $\phi$ in $\overline{\Omega}$ such that
$$-\Delta \phi -2b\nabla u\nabla\phi\geq \lambda \beta e^{\beta u}\phi.$$
We multiply  the previous inequality by $(e^{bu}-1)^{2\alpha}e^{2bu}/\phi$ for $\alpha>0$
and integrate by parts to obtain
\begin{eqnarray}\label{exp}
 \lambda \beta \int_{\Omega} e^{(\beta +2b)u}(e^{bu}-1)^{2\alpha} &
\leq &\\
\nonumber& \hspace{-11em}\leq & \hspace{-6em}
\int_{\Omega}\nabla\phi\nabla\left(\frac{e^{2bu}(e^{bu}-1)^{2\alpha}}{\phi}\right)-
 \int_{\Omega}
2b\frac{\nabla u\nabla\phi}{\phi}e^{2bu}(e^{bu}-1)^{2\alpha}\\
\nonumber &\hspace{-11em} = & \hspace{-6em}- \int_{\Omega}
\frac{|\nabla\phi|^2}{\phi^2}e^{2bu}(e^{bu}-1)^{2\alpha}  +\int_{\Omega} 2b\frac{\nabla
  u\nabla\phi}{\phi}e^{2bu}(e^{bu}-1)^{2\alpha}+\\
\nonumber &\hspace{-11em} &\hspace{-8em}+
{\ds \int_{\Omega} 2\alpha b\frac{\nabla
  u\nabla\phi}{\phi}e^{bu}e^{2bu}(e^{bu}-1)^{2\alpha-1}- \int_{\Omega}
2b\frac{\nabla u\nabla\phi}{\phi}e^{2bu}(e^{bu}-1)^{2\alpha}}\\
\nonumber& \hspace{-11em}= &\hspace{-6em} - \int_{\Omega}
\frac{|\nabla\phi|^2}{\phi^2}e^{2bu}(e^{bu}-1)^{2\alpha} +
\int_{\Omega} 2\alpha b\frac{\nabla u\nabla\phi}{\phi}e^{2bu}e^{bu}(e^{bu}-1)^{2\alpha-1}\\
\nonumber & \hspace{-11em}\leq &  \hspace{-6em}- \int_{\Omega}
\frac{|\nabla\phi|^2}{\phi^2}e^{2bu}(e^{bu}-1)^{2\alpha}+
\alpha^2b^2\int_{\Omega} |\nabla u|^2
e^{4bu}(e^{bu}-1)^{2\alpha-2}+ \\
\nonumber & & \hspace{2em}+ \int_{\Omega}
\frac{|\nabla\phi|^2}{\phi^2}e^{2bu}(e^{bu}-1)^{2\alpha}\\
\nonumber &\hspace{-11em}= & \hspace{-6em}\alpha^2b^2\int_{\Omega} |\nabla u|^2
e^{4bu}(e^{bu}-1)^{2\alpha-2}.
\end{eqnarray}

Using Lemma \ref{lemmaprop} with $\gamma=4$ we get, if $2\alpha>1$,
$$\int_{\Omega}|\nabla u|^2e^{4bu}(e^{bu}-1)^{2\alpha-2}dx\leq\frac{\lambda}{b(2\alpha+1)}\int_{\Omega}
  e^{(\beta+(2\alpha+2)b)u}dx+\lambda L_4,$$
where $L_4$ is a linear combination of the
 integrals $\int_{\Omega}e^{(\beta+2\alpha b)u}dx$ and \linebreak $\int_{\Omega}e^{(\beta+(2\alpha+1)b)u}dx$ with coefficients depending only on
$b$ and $\alpha$.

Therefore, replacing $(e^{bu}-1)^{2\alpha}$ by
$e^{2\alpha bu}$ on the left hand side of \eqref{exp} and combining
all the remaining terms with $L_4$ from above and denoting it by $L$, we obtain
\begin{equation}\label{contab+}
\lambda\beta\int_{\Omega}  e^{(\beta+(2\alpha+2)b)u}\leq \lambda\frac{\alpha^2 b}{2\alpha+1}\int_{\Omega}
  e^{(\beta+(2\alpha+2)b)u}+\lambda L.
\end{equation}
Note that $L$ represents a linear combination of integrals involving  the exponential
function $e^{\delta u}$ with exponent $\delta<\beta+(2\alpha+2)b$. Such terms
can be absorbed into the left hand side of \eqref{contab+}. In fact,
if $0<a_1<a_2$ then, for every $\epsilon>0$ there exists a constant
$C_{\epsilon}>0$ such that  $e^{a_1u}\leq\epsilon
e^{a_2u}+C_{\epsilon}$ for all $u\in(0,+\infty)$. Hence, for every
$\delta$ as above and every $\epsilon$ there exists $C_{\epsilon,\delta}$ such that
$$\int_{\Omega}e^{\delta u}\leq \epsilon\int_{\Omega}e^{(\beta+(2\alpha+2)b)u}+C_{\epsilon,\delta}|\Omega|$$

Thus, if $\alpha>1/2$ satisfies
$$\frac{\alpha^2 b}{2\alpha+1}<\beta,$$
then $e^{(\beta+(2\alpha+2)b)u}\in L^1(\Omega)$. Solving for $\alpha$
we get
\begin{equation}\label{alphabctt}
\frac{1}{2}<\alpha<\frac{\beta+\sqrt{\beta(\beta+b)}}{b}.
\end{equation}
This inequality is satisfied for some $\alpha$ since $\beta>b/8$. Therefore

$$e^{\beta u} \in L^q(\Omega) \quad {\rm for }\quad 1+3\frac{b}{\beta}<q=2\frac{(\alpha+1)}{\beta}+1 <3 + 
2\frac{b}{\beta}+2\sqrt{1+ \frac{b}{\beta}}.$$

Note that the function $v=e^{bu}-1$  defined at the beginning of this section 
is thus
in $L^{r_1=q\beta/b}(\Omega)$ and $v$ satisfies $-\Delta v=\lambda b(v+1)^p$ where
$p=1+\beta/b$. Therefore $(v+1)^p\in L^{r_1/p}$ and hence, using the
  equation for $v$, we have that $v\in W^{2,r_1/p}$. Since
  $W^{2,r}\subset L^s$ if $1/s=1/r-2/n$ we get that $v\in L^s$ for
  $s=(nr_1)/(pn-2r_1)$. If $s>r_1$, that is, 
$n< 2q$, then $v$ is bounded by an iterative procedure. 
Hence, $v\in L^{\infty}(\Omega)$ if
$$n < 6 + 4\frac{b}{\beta} +4\sqrt{1+\frac{b}{\beta}},$$
as we already knew from \eqref{vbctt}. This was totally expected since both results are
achieved using equivalent assumptions. 

 \vspace{1em}

Finally, as an example, consider the case $\beta=1$ and  $b\equiv 1$ in the
expression above. The equation for $u$ becomes
$$-\Delta u -|\nabla u|^2=e^u.$$ The stable solutions $u$ of this equation
satisfy
$$u\in L^{\infty}(\Omega) \quad \textrm{ if } \quad n<
10+4\sqrt{2},\quad \text{that is }\,n\leq 15,$$ with a uniform
$L^{\infty}$ bound as in Proposition \ref{propb+}.

\vspace{1em}

For another example let $\beta= 1$ and $b$ tend to $0$.  The equation becomes
$$-\Delta u = e^u$$ and the result above yields
$u\in L^{\infty}(\Omega)$ if and only if $n<10$,
which coincides with the result of \cite{CR}.

\begin{rem}{\rm We note here that by perturbing the equation
    $-\Delta u=e^u$ with a quadratic gradient term we actually obtain
    more regularity for stable solutions $u$.}
\end{rem}

\end{subsection}
\begin{subsection}{Case $b=ctt<0$}

In this case we use a modified
Hopf-Cole transformation  $v=1-e^{bu}$. If $u$ is bounded, the new function
$v$ is positive and bounded by $1$, that is, $0<v<1$ for $u$ bounded. We note here that $v=1$
corresponds to $u=\infty$. Moreover, $v$
satisfies

\begin{equation}\label{eqv-}
\left\{\begin{array}{rcll}
       -\Delta v &= & \lambda|b|(1-v)g\left(\frac{1}{b}\log(1-v)\right) & {\rm in} \; \Omega\\
               v & \geq & 0 &  {\rm in} \;  \Omega\\
               v & = & 0 & {\rm on} \;\partial\Omega,
       \end{array}
\right.
\end{equation}

We will again denote by $f$ the nonlinearity appearing on the right-hand
side of the equation above, that is, $v$ satisfies
$$-\Delta v =\lambda f(v) \text{ in }\Omega, \text{ where } f(v)=|b|(1-v)g\left(\frac{1}{b}\log(1-v)\right) .$$

Considering the same typical example as in the previous section, we let $g(u)=e^{\beta u}$ for
some constant $\beta>0$, 
$$-\Delta u -b|\nabla u|^2 = e^{\beta u}.$$  The equation for $v$ becomes
$$ -\Delta v = \lambda|b|(1-v)^p,$$
where  $p=1+\frac{\beta}{b}$. If $\beta> -b=|b|$ then $p<0$. This is the
case studied by Mignot and Puel \cite{MP} and more recently by
Esposito in \cite{E}. They prove that stable solutions $v$ satisfy

$$v<1\text{ in }\Omega \quad \textrm{ if } \quad  n<
2+\frac{4p}{p-1} + 4\sqrt{\frac{p}{p-1}},$$ with a bound for $v$ away
from $1$, uniform in $v$. In our case, for
$p=1+ \beta/b$ and $\beta>-b$ we have that

$$v<1 \text{ in }\Omega \quad \textrm{ if } \quad n<
6+ 4\frac{b}{\beta}+4\sqrt{1+\frac{b}{\beta}}.$$

That is,
\begin{prop}\label{propb-}
Let $b<0$ be a constant, $\beta>-b$ and $u$ a positive classical stable solution
to 
\begin{equation*}
\left\{\begin{array}{rcll}
-\Delta u -b|\nabla u|^2 & = & \lambda e^{\beta u}& \text{ in }
\Omega\\
u & = & 0& \text{ on }\partial\Omega,
\end{array}\right.
\end{equation*} where $\lambda>0$ a parameter. Then
$$||u||_{L^{\infty}(\Omega)}\leq C \quad \textrm{ if } \quad n<
6+4\frac{b}{\beta}+4\sqrt{1+\frac{b}{\beta}},$$ where $C$ is a
constant depending only on $n, b, \beta$ and $\Omega$ (in particular
is independent of $\lambda$).
\end{prop}

It is a nontrivial fact to note that, since $b<0$, this result yields
less regularity
for stable solutions
than the one obtained for $b>0$ (recall Proposition \ref{propb+}). 
Note that the condition on
the exponent $\beta$ is more restrictive than the one we have in the
case of $b>0$. For example, 
if we consider the case $b=-1$,
that is, if $u$ satisfies the equation
$$-\Delta u +|\nabla u|^2 = e^{\beta u},$$ we find the assumption
$\beta >1$ in Proposition \ref{propb-},
which means that we can not apply the previous result to the equation
$-\Delta u + |\nabla u|^2=e^u$. Nevertheless, for this particular
case, the equation for $v$ would be a linear Poisson equation $-\Delta
v=\lambda |b|$, and therefore $u$ would be bounded for all dimensions.

\end{subsection}
\end{section}

\begin{section}{General $b(x)$}

In this section we study the case of a general bounded function
$b=b(x)$. Let $u$ be a positive solution to
the equation
\begin{equation}\label{eqb}
-\Delta u -b(x)|\nabla u|^2= \lambda g(u)
\end{equation}
with Dirichlet boundary conditions.
We denote by $\bsub$ and $\bsup$ the infimum and the supremum of
$b(x)$ respectively,  that is,
$$\bsub\leq b(x)\leq \bsup \quad \text{for every}\; x\in\Omega.$$

Equation ($\ref{eqb}$) can no longer be transformed into a classical one, and
there are no known regularity results for stable solutions. Following the
computations we introduced in the previous sections we will study this
equation directly, only with the assumptions on $u$, that is, $u$ is
stable as defined in Definition $\ref{stable5}$.

We consider the special case where $g(u)=e^u$. Then, there exists
$\phi>0$ in $\overline{\Omega}$ such that
\begin{equation}\label{propphi}
-\Delta\phi -2b(x)\nabla u\nabla\phi \geq  \lambda e^u\phi.
\end{equation}
Note that we restrict ourselves to the case $g(u)=e^{\beta u}$ with
$\beta=1$ since, if $\beta\neq 1$, 
a change of variables would lead to the equation with $\beta=1$, only with a
different parameter $\lambda$ and a different $b=b(x)$.

The first result that we prove is the following.

\begin{prop}\label{propgenb}
Let $b=b(x)$ be a bounded function such that $\bsub\leq b(x)\leq\bsup$
for some constants $\bsub$ and  $\bsup$ with $\bsup>0$
and $u$ a positive classical
stable solution to
\begin{equation*}
\left\{\begin{array}{rcll}
-\Delta u -b(x)|\nabla u|^2 & =& \lambda e^u & \text{ in }\Omega\\
u & = & 0 & \text{ on }\partial\Omega,\\
\end{array}\right.
\end{equation*}
 where $\Omega\subset\R^n$ is a smooth bounded domain
and $\lambda>0$ is a parameter.
Then, for every positive constants $\delta$ and $\eta$ with
$\delta^2+\eta^2\leq 1$, if $(\bsup-\bsub)<\frac{\delta^2}{\eta^2}(\eta^2-\frac{\bsup}{8})$,
$$||e^{u}||_{L^q(\Omega)}\leq C \quad {\rm for }\quad q <1 + 2\eta^2+
2\bsup +2\sqrt{\eta^2(\eta^2+\bsup)-2\bsup(\bsup-\bsub)\frac{\eta^2}{\delta^2}},$$
where $C$ depends only on $n, b$ and $\Omega$ (in particular is
independent of $\lambda$). 
 \end{prop}

We note that we have made
 no assumptions on the sign of the function $b(x)$. In fact,
 the only condition we have on $b(x)$ is the oscillation condition
 involving $\bsub$ and $\bsup$, $(\bsup-\bsub)<\frac{\delta^2}{\eta^2}(\eta^2-\frac{\bsup}{8})$. This condition guarantees that the expression
 inside the square root above is nonnegative. 

\vspace{1em}

In the case that $b>0$ is constant we have that $b\equiv\bsup=\bsub>0$
and hence $\bsup-\bsub=0$ and we may choose $\eta\uparrow 1$ and
$\delta\downarrow 0$ to obtain
$$e^u\in L^q(\Omega),\quad {\rm for} \quad q< 3+2b+2\sqrt{b+1},$$
if $b<8$, which coincides with the result of Proposition $\ref{propb+}$ ($\beta=1$). 

The same regularity result still holds if $b$ has oscillation of order
$\epsilon$, since we may choose $\delta^2=2\bsup\sqrt{\bsup-\bsub}$
which is again of order $\epsilon$ and therefore we can let $\eta$
tend to $1$ and $\delta$ tend to $0$. 
\vspace{1em}

As before, we begin by establishing the following estimate:

\begin{lema}\label{lemmapropgenb}
Let $b(x)\leq \bsup$ with $\bsup>0$ and $u$ be a positive solution to 
\begin{equation*}
\left\{\begin{array}{rcll}
-\Delta u -b(x)|\nabla u|^2& = & \lambda e^{u}&\text{ in } \Omega\\
u & = & 0&\text{ on }\partial\Omega,
\end{array}\right.
\end{equation*} where $\lambda>0$ is a parameter and
$\beta>0$. For $\gamma\in {\mathbb N}$ satisfying $\gamma\geq 2$ have
\begin{equation}\label{intlemmagenb}
\int_{\Omega}|\nabla u|^2e^{\gamma
  \bsup u}(e^{\bsup u}-1)^{2\alpha-2}dx\leq\frac{\lambda}{b(2\alpha+\gamma-3)}\int_{\Omega}
  e^{(1+(2\alpha+\gamma-2)\bsup )u}dx+\lambda L_{\gamma},
\end{equation} 
where $\alpha>1/2$ is a parameter and $L_{\gamma}$ is a linear combination of the
$\gamma-2$ integrals $\int_{\Omega}e^{(\beta+(2\alpha+k)b)u}dx$,
$k=0,1,\dots, \gamma-3$, with coefficients depending only on
$b,\alpha$ and $\gamma$.
\end{lema}

\begin{proof}
Let $2\alpha>1$ and $\gamma\geq 2$ an integer. We have,
\begin{eqnarray*}
 \hspace{-1em}{\ds \int_{\Omega}|\nabla u|^2e^{\gamma \bsup u}(e^{\bsup u}-1)^{2\alpha-2}} &  \hspace{-1em}= 
& \hspace{-1em}{\ds  \int_{\Omega} \nabla u e^{(\gamma-1)\bsup u}\nabla ue^{\bsup u}(e^{\bsup u}-1)^{2\alpha-2}}\\
&  \hspace{-1em}=&  \hspace{-1em}{\ds  \int_{\Omega} \nabla u e^{(\gamma-1)\bsup u}\frac{\nabla (e^{\bsup u}-1)^{2\alpha-1}}{b(2\alpha-1)}}\\
& \hspace{-1em} = &  \hspace{-1em}{\ds\int_{\Omega} \frac{
  e^{(\gamma-1)\bsup u}(e^{\bsup u}-1)^{2\alpha-1}}{\bsup(2\alpha-1)}\left(-\Delta
u-(\gamma-1)\bsup|\nabla u|^2\right).}
\end{eqnarray*}
Using the equation for $u$ and the fact that $b(x)\leq \bsup$ with $\bsup>0$ we have
\begin{eqnarray*}
\hspace{-1em}{\ds \int_{\Omega}|\nabla u|^2e^{\gamma \bsup u}(e^{\bsup u}-1)^{2\alpha-2}} &\hspace{-1em} = 
&\hspace{-1em} {\ds \int_{\Omega} \frac{
  e^{(\gamma-1)\bsup u}(e^{\bsup u}-1)^{2\alpha-1}}{\bsup(2\alpha-1)}\left(\lambda e^{u}
-(\gamma-2)\bsup|\nabla u|^2\right)}+ \\
& & \hspace{2em} + \int_{\Omega} \frac{
  e^{(\gamma-1)\bsup u}(e^{\bsup u}-1)^{2\alpha-1}}{\bsup(2\alpha-1)}(b-\bsup)|\nabla u|^2\\
&\hspace{-1em} \leq & {\ds  \frac{\lambda}{\bsup(2\alpha-1)}\int_{\Omega}
e^{(1+(2\alpha+\gamma-2)\bsup) u}-}\\
& &{\ds \hspace{-12em}-\frac{\gamma-2}{2\alpha-1}\int_{\Omega}|\nabla u|^2e^{\gamma
  \bsup u}(e^{\bsup u}-1)^{2\alpha-2}+ \frac{\gamma-2}{2\alpha-1}\int_{\Omega}|\nabla u|^2e^{(\gamma-1)\bsup u}(e^{\bsup u}-1)^{2\alpha-2}.}
 \end{eqnarray*}
This yields, adding the left hand side to the second term on the right
hand side, and since $2\alpha>1$ and $\gamma\geq 2$ (and thus
$2\alpha>\gamma-3$),
\begin{eqnarray}\label{gammagenb}
\int_{\Omega}|\nabla u|^2e^{\gamma \bsup u}(e^{\bsup u}-1)^{2\alpha-2}& \leq &
 \frac{\lambda}{\bsup(2\alpha+\gamma-3)}\int_{\Omega}
e^{(1+(2\alpha+\gamma-2)\bsup)u}+\\
\nonumber& & + \frac{\gamma-2}{2\alpha+\gamma-3}\int_{\Omega}|\nabla u|^2e^{(\gamma-1)\bsup u}(e^{\bsup u}-1)^{2\alpha-2}.
\end{eqnarray}
If $\gamma=2$ the second term on the right hand side of
\eqref{gammagenb} is zero and we conclude \eqref{intlemmagenb} (as
desired) with $L=0$. Otherwise, for $\gamma\in{\mathbb N}$,
$\gamma\geq 3$, we may repeat the computations
above with $\gamma$ replaced by $\gamma-1$ to obtain
\begin{eqnarray*}
\int_{\Omega}|\nabla u|^2e^{(\gamma-1)\bsup u}(e^{\bsup u}-1)^{2\alpha-2}& \leq &
 \frac{\lambda}{\bsup(2\alpha+\gamma-4)}\int_{\Omega}
e^{(1+(2\alpha+\gamma-3)\bsup)u}+\\
& & + \frac{\gamma-3}{2\alpha+\gamma-4}\int_{\Omega}|\nabla u|^2e^{(\gamma-2)\bsup u}(e^{\bsup u}-1)^{2\alpha-2}.
\end{eqnarray*}
Note that the left hand side is the second integral on the right hand
side of \eqref{gammagenb}, which remains to be controlled. Note also that the exponent of the exponential function in the first
integral on the right hand side decreases on each iteration. We may
continue this process until we are left with the integral
 of $|\nabla u|^2e^{2\bsup u}(e^{\bsup u}-1)^{2\alpha-2}$. For this integral,
 $$\int_{\Omega} |\nabla
 u|^2e^{2\bsup u}(e^{\bsup u}-1)^{2\alpha-2}\leq\frac{\lambda}{\bsup(2\alpha-1)}\int_{\Omega}
 e^{(1+2\alpha \bsup)u}.$$
Hence,
\begin{equation*}
\int_{\Omega}|\nabla u|^2e^{\gamma
  \bsup u}(e^{\bsup u}-1)^{2\alpha-2}dx\leq\frac{\lambda}{\bsup(2\alpha+\gamma-3)}\int_{\Omega}
  e^{(1+(2\alpha+\gamma-2)\bsup)u}dx+\lambda L_{\gamma},
\end{equation*} 
where $L_{\gamma}$ is a linear combination of the
$\gamma-2$ integrals $\int_{\Omega}e^{(\beta+(2\alpha+k)b)u}dx$ for
$k=0,1,\dots, \gamma-3$, with coefficients depending only on
$b,\alpha$ and $\gamma$.
\end{proof}

We now prove the proposition. We follow the computations as in the
proof of Proposition \ref{propb+}.
\begin{proof}
The assumption we have made on $u$ is that it is stable according
to Definition $\ref{stable5}$, that is, there exists a positive
function $\phi$ in $\overline{\Omega}$ such that
$$-\Delta \phi -2b(x)\nabla u\nabla\phi\geq \lambda e^{ u}\phi.$$
We multiply  the previous inequality by $(e^{\bsup u}-1)^{2\alpha}e^{2\bsup u}/\phi$ for $\alpha>0$
and integrate by parts to obtain
\begin{eqnarray}\label{expgenb}
 \lambda  \int_{\Omega} e^{(1 +2\bsup)u}(e^{\bsup u}-1)^{2\alpha} &
\leq &\\
\nonumber& \hspace{-13em}\leq & \hspace{-7em}
\int_{\Omega}\nabla\phi\nabla\left(\frac{e^{2\bsup u}(e^{\bsup u}-1)^{2\alpha}}{\phi}\right)-
 \int_{\Omega}
2b\frac{\nabla u\nabla\phi}{\phi}e^{2\bsup u}(e^{\bsup u}-1)^{2\alpha}\\
\nonumber &\hspace{-13em} = & \hspace{-7em}- \int_{\Omega}
\frac{|\nabla\phi|^2}{\phi^2}e^{2\bsup u}(e^{\bsup u}-1)^{2\alpha}  +\int_{\Omega} 2\bsup\frac{\nabla
  u\nabla\phi}{\phi}e^{2\bsup u}(e^{\bsup u}-1)^{2\alpha}+\\
\nonumber &\hspace{-13em} &\hspace{-9em}+
{\ds \int_{\Omega} 2\alpha \bsup\frac{\nabla
  u\nabla\phi}{\phi}e^{\bsup u}e^{2\bsup u}(e^{\bsup u}-1)^{2\alpha-1}- \int_{\Omega}
2b\frac{\nabla u\nabla\phi}{\phi}e^{2\bsup u}(e^{\bsup u}-1)^{2\alpha}}\\
\nonumber& \hspace{-13em}= &\hspace{-7em} - \int_{\Omega}
\frac{|\nabla\phi|^2}{\phi^2}e^{2\bsup u}(e^{\bsup u}-1)^{2\alpha} +
\int_{\Omega} 2\alpha \bsup\frac{\nabla u\nabla\phi}{\phi}e^{2\bsup u}e^{\bsup u}(e^{\bsup u}-1)^{2\alpha-1}+\\
\nonumber & \hspace{-13em} &  + \int_{\Omega} 2(\bsup-b)\frac{\nabla u\nabla\phi}{\phi}e^{2\bsup u}(e^{\bsup u}-1)^{2\alpha}\\
\nonumber & \hspace{-13em}\leq &  \hspace{-7em}(\delta^2+\eta^2- 1)\int_{\Omega}
\frac{|\nabla\phi|^2}{\phi^2}e^{2\bsup u}(e^{\bsup u}-1)^{2\alpha}+\\
\nonumber & &\hspace{-10em}+\frac{2\bsup(\bsup-\bsub)}{\delta^2}\int_{\Omega} |\nabla u|^2
e^{2\bsup u}(e^{\bsup u}-1)^{2\alpha}+  \frac{\alpha^2\bsup^2}{\eta^2}\int_{\Omega} |\nabla u|^2
e^{4\bsup u}(e^{\bsup u}-1)^{2\alpha-2} \\
\nonumber &\hspace{-13em}= & \hspace{-7em}(\delta^2+\eta^2- 1)\int_{\Omega}
\frac{|\nabla\phi|^2}{\phi^2}e^{2\bsup u}(e^{\bsup u}-1)^{2\alpha}+\\
\nonumber & &\hspace{-2em}+ \left(\frac{2\bsup(\bsup-\bsub)}{\delta^2}+\frac{\alpha^2\bsup^2}{\eta^2}\right)\int_{\Omega} |\nabla u|^2
e^{4\bsup u}(e^{\bsup u}-1)^{2\alpha-2},
\end{eqnarray}
where $\delta>0$ and $\eta>0$ are constants.
Using Lemma \ref{lemmapropgenb} with $\gamma=4$ we get, for $\alpha>1/2$,
$$\int|\nabla u|^2e^{4\bsup u}(e^{\bsup u}-1)^{2\alpha-2}dx\leq\frac{\lambda}{\bsup(2\alpha+1)}\int
  e^{(1+(2\alpha+2)\bsup)u}dx+\lambda L_4,$$
where $L_4$ is a linear combination of the two
integrals $\int_{\Omega}e^{(\beta+(2\alpha)b)u}dx$ and
  $\int_{\Omega}e^{(\beta+(2\alpha+1)b)u}dx$ with coefficients depending only on
$b$ and $\alpha$.

Therefore, replacing $(e^{\bsup u}-1)^{2\alpha}$ by
$e^{2\alpha \bsup u}$ on the left hand side of \eqref{exp} and combining
all the remaining terms with $L_4$ from above and denoting it by $L$, we obtain
\begin{eqnarray}\label{contagenb}
\lambda\int  e^{(1+(2\alpha+2)\bsup)u}& \leq & 
(\delta^2+\eta^2- 1)\int
\frac{|\nabla\phi|^2}{\phi^2}e^{2\bsup u}(e^{\bsup u}-1)^{2\alpha}+\\
\nonumber & &\hspace{-2em}+ \left(\frac{2\bsup(\bsup-\bsub)}{\delta^2}+\frac{\alpha^2\bsup^2}{\eta^2}\right)\frac{\lambda}{\bsup(2\alpha+1)}\int
  e^{(1+(2\alpha+2)\bsup)u}+\lambda L.
\end{eqnarray}

Note that $L$ represents a linear combination of integrals involving  the exponential
function $e^{\delta u}$ with exponent $\delta<\beta+(2\alpha+2)\bsup$. Such terms
can be absorbed into the left hand side of \eqref{contagenb}. In fact,
if $0<a_1<a_2$ then, for every $\epsilon>0$ there exists a constant
$C_{\epsilon}>0$ such that  $e^{a_1u}\leq\epsilon
e^{a_2u}+C_{\epsilon}$ for all $u\in(0,+\infty)$. Hence, for every
$\delta$ as above and every $\epsilon$ there exists $C_{\epsilon,\delta}$ such that
$$\int_{\Omega}e^{\delta u}\leq \epsilon\int_{\Omega}e^{(\beta+(2\alpha+2)\bsup)u}+C_{\epsilon,\delta}|\Omega|$$

Thus, if $\delta^2+\eta^2\leq 1$ and $\alpha>1/2$ satisfies
$$\left(\frac{2\bsup(\bsup-\bsub)}{\delta^2}+\frac{\alpha^2\bsup^2}{\eta^2}\right)\frac{1}{\bsup(2\alpha+1)}<1,$$
then $e^{(1+(2\alpha+2)\bsup)u}\in L^1(\Omega)$. Solving for $\alpha$
we get
\begin{equation}\label{alphagenb}
\frac{1}{2}<\alpha<\frac{\eta^2+\sqrt{\eta^2(\eta^2+\bsup)-2\bsup(\bsup-\bsub)\frac{\eta^2}{\delta^2}}}{\bsup}.
\end{equation}
This inequality is satisfied for some $\alpha$ since
$(\bsup-\bsub)<\frac{\delta^2}{\eta^2}(\eta^2-\frac{\bsup}{8})$. Therefore
\begin{equation}\label{qbx}
||e^{u}||_{L^q(\Omega)}\leq C \quad {\rm for }\quad q<1 + 2\eta^2+
2\bsup +2\sqrt{\eta^2(\eta^2+\bsup)-2\bsup(\bsup-\bsub)\frac{\eta^2}{\delta^2}},
\end{equation}
where $C$ is independent of $\lambda$.
\end{proof}

\begin{rem}\label{notestable}
We note that we can perform all the computations if we assume only that there exists
a function $\phi_{\epsilon}$, positive in $\overline{\Omega}$, such that
$$-\Delta\phi_{\epsilon} -2b(x)\nabla u\nabla\phi_{\epsilon}\geq
(\lambda -\epsilon)e^u\phi_{\epsilon},$$
for some small $\epsilon>0$. Letting $\epsilon$ tend to $0$ we obtain
the result above with the constant $C$ independent of $\epsilon$.
\end{rem}

\end{section}

\begin{section}{Further regularity for $b(x)\geq 0$}\label{sctreg}

In the case where the function $b$ is non-negative, we can reach
further regularity and prove a similar result to the one where $b$ is
constant, even though we are not able to transform the equation into a
classical one.  Using a well chosen Hopf-Cole
transformation we can construct a subsolution of the
classical equation with a power nonlinearity. Using a bootstrap
argument and Proposition \ref{propgenb}
 this is enough to
conclude about the regularity of $u$.

\begin{prop}\label{propreg}
Let $b(x)\geq 0$ and $0\leq \bsub\leq b(x)\leq\bsup$ in $\Omega$ for
some constants $\bsub$ and $\bsup$, and $u$ a positive classical stable
solution of 
\begin{equation*}
\left\{\begin{array}{rcll}
-\Delta u-b(x)|\nabla u|^2& =& \lambda e^u& \text{ in }\Omega\\
u& = &0 &\text{ on }\partial\Omega,
\end{array}\right.
\end{equation*} and $\lambda>0$ a parameter. 
For every positive constants $\delta$ and $\eta$ with
$\delta^2+\eta^2\leq 1$, let
$(\bsup-\bsub)<\frac{\delta^2}{\eta^2}(\eta^2-\frac{\bsup}{8})$ and $n <2 + 4\eta^2+4\bsup
+4\sqrt{\eta^2(\eta^2+\bsup)-2\bsup(\bsup-\bsub)\frac{\eta^2}{\delta^2}}.$
Then, $||u||_{L^{\infty}(\Omega)}\leq C$, where $C$ depends only on $n, b$
and $\Omega$.
\end{prop}

\begin{proof}

Consider the Hopf-Cole transformation $v=e^{\bsup u}-1$. We have that
\begin{eqnarray*}
-\Delta v & = & \bsup e^{\bsup u}\left(-\Delta u - \bsup|\nabla u|^2\right)\\
& =& \bsup e^{\bsup u}\left(\lambda e^u+ (b(x)-\bsup )|\nabla u|^2\right)\\
& \leq & \lambda\bsup e^{\bsup u}e^u\\
&= & \lambda\bsup (v+1)^{\frac{\bsup +1}{\bsup }}
\end{eqnarray*}
which means that
 $v$ is a positive subsolution of the classical equation, i.e., 
$$-\Delta v \leq \lambda\bsup (v+1)^p \quad \text{ in } \Omega,$$ for $p=(\bsup +1)/\bsup $. Let $w$ be the
solution  to the linear problem 
$$\left\{\begin{array}{rcll}
-\Delta w & = & \lambda\bsup (v+1)^p & \text{ in }\Omega\\
w & = & v &  \text{ in }\partial\Omega.
\end{array}\right.$$ Then,
trivially $-\Delta v\leq-\Delta w$ and hence, by the maximum principle
$$0\leq v\leq w.$$
If $v\in L^s(\Omega)$
then, using the equation for $w$ we get that $w\in W^{2,s/p}(\Omega)\subset L^r(\Omega)$ for $r=(ns)/(np-2q)$. Now, $r>s$ if
$n<2s/(p-1)$. Therefore, by a bootstrap argument, $w$ and hence $v$ is in
$L^{\infty}(\Omega)$ if $n<2s/(p-1)$, that is, $n< 2\bsup s$. 

\vspace{1em}

From the previous section we know that $e^u\in L^q(\Omega)$ for $q$ given by
($\ref{qbx}$). Given the definition of $v$ we get that $v\in L^{q/\bsup}(\Omega)$,
i.e., we can replace $s=q/\bsup$ in the discussion above. Thus we obtain
that $v$ and hence $u$ are in $L^{\infty}(\Omega)$ if $n< 2q$, that is,
$$u\in L^{\infty}(\Omega)\quad {\rm if }\quad  n <2 + 4\eta^2+
4\bsup+4\sqrt{\eta^2(\eta^2+\bsup)-2\bsup(\bsup-\bsub)\frac{\eta^2}{\delta^2}},$$
with $\delta^2+\eta^2\leq 1$.
\end{proof}

\end{section}

\begin{section}{Existence for $b(x)\geq 0$}

In this section we  prove an existence theorem, 
in terms of $\lambda$, of solutions to the
problem
\begin{equation}\label{lambdau}
\left\{\begin{array}{rcll}
 -\Delta u-b(x)|\nabla u|^2& = & \lambda g(u) & {\rm in}\; \Omega \\
                    u & \geq & 0 & {\rm in}\; \Omega\\
                    u & = & 0 & {\rm on }\;\partial\Omega,
\end{array}\right.
\end{equation}
where $g$ is a nonlinearity with assumptions to be detailed later,
 $b(x)\geq
0$ and $\Omega\subset\R^n$ is a smooth bounded domain with $n\geq 2$. 

If $b(x)=b$ is constant, using the Hopf-Cole transformation we
reach an equation for $v$ of the form
\begin{equation}\label{lambdav}
\left\{\begin{array}{rcll}
-\Delta v & = &  \lambda f(v) & \text{ in } \Omega\\
v & = & 0 & \text{ in }\partial\Omega.
\end{array}\right.
\end{equation}
 where $f(v)=b(v+1)g\left(\frac{1}{b}\ln(v+1)\right).$ 
Equations of the type ($\ref{lambdav}$) have been extensively studied.
Under the following conditions on $f$:
\begin{equation}\label{condf}
f \; {\rm is } \; C^1,  \text{ convex, nondecreasing },\; f(0)>0 \; {\rm and }\;
\lim_{v\rightarrow +\infty}\frac{f(v)}{v}=+\infty,
\end{equation} there exists a finite parameter $\lambda^*>0$ such
that, for $\lambda>\lambda^*$ there is no bounded solution to
($\ref{lambdav}$). On the other hand, for $0<\lambda<\lambda^*$ there
exists a minimal bounded solution $v_{\lambda}$, where minimal means
smallest.

These conditions hold for $f$ if we assume that $g$ satisfies:
\begin{equation}\label{condg2}
g \; {\rm is } \; C^1,  \text{convex, nondecreasing },\; g(0)>0 \; {\rm and }\;
\lim_{u\rightarrow +\infty}g(u)=+\infty.
\end{equation}

In the general case for a non-negative function $b$ we prove the
following theorem.

 \begin{teo}\label{prop1}
Let $b=b(x)\geq 0$ be a $C^{\alpha}(\overline{\Omega})$ function defined in a smooth bounded domain $\Omega\subset\R^n$  and
$g$ be a nondecreasing $C^1$ function with $g(0)>0$ and
$\lim_{u\rightarrow+\infty}\frac{g(u)}{u}=+\infty$. Then, there
exists a parameter $0<\lambda^{*}<\infty $ such that:

\noindent {\rm (a)} If $\lambda > \lambda^{*}$ then there is no
classical solution of
$(\ref{lambdau})$.\\
{\rm (b)} If $0\leq \lambda < \lambda^{*}$ then there exists a minimal
classical solution $u_\lambda$ of $(\ref{lambdau})$.
Moreover, $u_\lambda<u_\mu$ if $\lambda<\mu<\lambda^*$. 

In addition, if $g(u)=e^u$ and for every positive constants $\delta$
and $\eta$ with $\delta^2+\eta^2\leq 1$, the function $b$ satisfies $0\leq\bsub\leq b
\leq\bsup$ in $\Omega$ for constants $\bsub$ and $\bsup$ such that 
$(\bsup-\bsub)<\frac{\delta^2}{\eta^2}(\eta^2-\frac{\bsup}{8})$ and 
$n <2 + 4\eta^2+4\bsup
+4\sqrt{\eta^2(\eta^2+\bsup)-2\bsup(\bsup-\bsub)\frac{\eta^2}{\delta^2}}$,
then $u_\lambda$ is semi-stable. Moreover, the limit
$u^*=\lim_{\lambda\rightarrow\lambda^*}u_{\lambda}$ is a weak solution
of ($\ref{lambdau}$) for $\lambda=\lambda^*$. That is, it satisfies
$$-\int_{\Omega}u^*\Delta\xi -\int_{\Omega}b(x)|\nabla u^*|^2 \xi= 
\lambda^*\int_{\Omega}e^{u^*}\xi,$$
for every $\xi\in C^2(\overline{\Omega})$ with $\xi= 0$ on
$\partial\Omega$. In addition, the estimates of Proposition
$\ref{propreg}$ apply to $u^*$.
\end{teo}

\begin{proof}
First, we prove that there is no classical solution for large
$\lambda$. Let $u_{\lambda}$ be a bounded solution corresponding
to $\lambda$. Then, since $b\geq 0$ this function $u_{\lambda}$
is a supersolution of the classical problem
\begin{equation*}
\left\{\begin{array}{rcll}
-\Delta u& \geq &\lambda g(u) & \text{ in }\Omega\\
u & = & 0 & \text{ on }\partial\Omega.
\end{array}\right.
\end{equation*}
Since $g(0)>0$, $\underline{u}=0$ is a strict subsolution for every
$\lambda>0$. This would imply the existence of a classical solution
corresponding to $\lambda$ between $0$ and our supersolution
$u_{\lambda}$. We know this is only possible for $\lambda$ smaller
than a finite extremal parameter,
hence the same applies for the solutions to our problem
$(\ref{lambdau})$.

\vspace{1em}

Next, we prove the existence of a classical solution of ($\ref{lambdau}$)
for small $\lambda$.
For general $\lambda$, the existence of a bounded supersolution  
implies the existence of 
a minimal (smallest) classical solution $u_\lambda$. This solution is obtained by monotone iteration
starting from $0$. That is, $u_\lambda$ is the increasing limit of
$u_m$ where the functions $u_m$ are defined as $u_0\equiv 0$ and, for
$m\geq 1$
\begin{equation}
\left\{ \begin{array}{rcll}
-\Delta u_m-b(x)|\nabla u_m|^2& = & \lambda g(u_{m-1}) & {\rm in}\;
\Omega\\
u_m& = &0 & {\rm on }\;\partial\Omega.
\end{array}\right.
\end{equation}
The equation for $u_m$ may be written as 
$$-\Delta u_m= F(x,u_m,\nabla u_m)$$ where $F$ satisfies
$$|F(x,u_m,\xi)|\leq K(1+|\xi|^2),$$ for some constant $K$, since $b$
is a bounded function and, at step $m$, the function $u_{m-1}$ is
known and bounded. For this equation and under such conditions on $F$ we have
existence of solution $u_m\in W^{2,p}(\Omega)$ for every $p>1$ (see \cite{ACr}) and this
implies, for $p$ large, that $u_m\in C^{1,\alpha}(\overline{\Omega})$. Moreover $C^{1,\alpha}(\overline{\Omega})$ is compactly
embedded in $C^1(\overline{\Omega})$.

\vspace{1em}

We will prove by induction that this sequence $u_m$ is increasing. For
$m=1$ we have that
$$-\Delta u_1-b(x)|\nabla u_1|^2=\lambda g(0)>0=-\Delta
u_0-b(x)|\nabla u_0|^2,$$ which implies, for $b(x)\geq 0$ and since
$u_0\equiv 0$,
$$-\Delta u_1> -\Delta u_0.$$
By the classical maximum principle we have $u_1\geq u_0.$

Now assume $u_m\geq u_{m-1}$. Then,
\begin{eqnarray*}
-\Delta u_{m+1}-b(x)|\nabla u_{m+1}|^2 & = & \lambda g(u_m) \\
& \geq & \lambda g(u_{m-1})\\
& = & -\Delta u_m-b(x)|\nabla u_m|^2,
\end{eqnarray*}
where we have used that $g$ is nondecreasing. Let
$w=u_{m+1}-u_m$. From the previous inequality we derive an inequality
satisfied by $w$. Namely,
$$-\Delta w - \vec{B}(x)\cdot\nabla w\geq 0,$$
where $\vec{B}(x)=b(x)\nabla (u_{m+1}+u_m).$ By the maximum principle
we have that $w\geq 0$, that is, $$u_{m+1}\geq u_m.$$
Therefore we have constructed an increasing sequence $u_m$. 

\vspace{1em}

Let now $\usup$ be the solution of
\begin{equation}
\left\{ \begin{array}{rl}
 -\Delta \overline{u} -b(x)|\nabla\usup|^2=1 & \textrm{in } B_{1} \\
 \overline{u}=0 & \textrm{on }\partial B_{1}.
\end{array}\right.
\end{equation}
This function $\usup$ is a bounded supersolution of ($\ref{lambdau}$) 
for small $\lambda$,
whenever  $\lambda g(\max \overline{u}) < 1$.

Using induction and the maximum principle as above we can prove that
the sequence is bounded by $\usup$, i.e., 
$$u_0\leq u_1\leq \dots \leq u_m\leq u_{m+1}\leq \dots \leq \usup.$$
This implies there exists a limit,
$$u_{\lambda}:=\lim_{m\rightarrow\infty} u_m,$$
and moreover, $u_{\lambda}$ is a solution to ($\ref{lambdau}$). In fact, since
$u_m\in W^{2,p}(\Omega)$ we get that, for $p$ large, $u_m\in C^{1,\alpha}(\overline{\Omega})$. Using the
equation and the fact that $b\in C^{\alpha}(\overline{\Omega})$, we get that $u_m\in C^{2,\alpha}(\overline{\Omega})$ and hence converges to a solution of ($\ref{lambdau}$).

The extremal parameter $\lambda^*$ is now defined as the supremum of
all $\lambda>0$ for which $(\ref{lambdau})$ admits a classical solution. Hence,
both $0<\lambda^*<\infty$ and  part (a) of the proposition holds.

(b) Next, if $\lambda<\lambda^*$ there exists $\mu$ with
$\lambda<\mu<\lambda^*$ and such that $(\ref{lambdau})$ admits a classical
solution $u_{\mu}$. Since $g>0$, $u_{\mu}$ is a bounded supersolution of
$(\ref{lambdau})$, and hence the same monotone iteration argument used above
shows that  $(\ref{lambdau})$ admits a classical solution $u_\lambda$ with
$u_\lambda\le u$.
In addition, we have shown that
$u_\lambda$ is smaller than any classical supersolution of $(\ref{lambdau})$.
It follows that $u_\lambda$ is minimal (i.e., the smallest solution)
and that $u_\lambda<u_\mu$.

\vspace{1em}

Consider now the case where $g(u)=e^u$, and assume that for every positive
constants $\delta$ and $\eta$ with $\delta^2+\eta^2\leq 1$ we have
 $$(\bsup-\bsub)<\frac{\delta^2}{\eta^2}(\eta^2-\frac{\bsup}{8})\;
\text{ and }\; n <2 + 4\eta^2+4\bsup
+4\sqrt{\eta^2(\eta^2+\bsup)-2\bsup(\bsup-\bsub)\frac{\eta^2}{\delta^2}}.$$
First we prove that
$u_{\lambda}$ is semi-stable, meaning by semi-stable that the first
eigenvalue $\lambda_1$ of the linearized operator $L_{\lambda}$ is
non-negative. That is, 
$$\lambda_1(L_{\lambda})\geq 0\quad \text{ where } \; L_{\lambda}:=-\Delta -2b(x)\nabla u_{\lambda}\nabla -\lambda
e^{u_{\lambda}}.$$
We have seen that $u_{\lambda}$ forms an increasing sequence with
respect to $\lambda$. For $\delta>0$  let
$v_{\delta}=u_{\lambda+\delta}-u_{\lambda}>0$. Using the equations for
$u_{\lambda+\delta}$ and $u_{\lambda}$ we have that $v_{\delta}$
satisfies 
$ -\Delta v_{\delta} -2b(x)\nabla
\left(\frac{u_{\lambda+\delta}+u_{\lambda}}{2}\right)\nabla v_{\delta}
- \lambda e^{\eta}v_{\delta} >0,$
where $\eta$ is between $u_{\lambda}$ and $u_{\lambda+\delta}$ and we
have used that $\delta>0$. Therefore, if we define the linear operator
$$L_{\lambda,\delta}:=-\Delta
-2b(x)\nabla\left(\frac{u_{\lambda+\delta}+u_{\lambda}}{2}\right)\nabla
-\lambda e^{\eta},$$we have that, at $v_{\delta}$,
$$L_{\lambda,\delta} v_{\delta} = -\Delta v_{\delta} -2b(x)\nabla
\left(\frac{u_{\lambda+\delta}+u_{\lambda}}{2}\right)\nabla v_{\delta}
- \lambda e^{\eta}v_{\delta} >0 \;\text{ in }\;\Omega,$$
and thus $v_{\delta}$ is a  strict supersolution positive in $\Omega$ of
$L_{\lambda,\delta} = 0$ in $\Omega$
and hence $\lambda_1(L_{\lambda,\delta})>0$. 

Now we  pass to the limit in $\delta$ and obtain that
$\lambda_1(L_{\lambda})\geq 0$, that is, $u_{\lambda}$ is semi-stable
as defined above. This can
be done using Propositions 2.1 and 5.1  of \cite{BNV} which establishes that, for
bounded coefficients, 
$\lambda_1$ is Lispchitz continuous with respect to both the first and
the zeroth order coefficients. 

For every $\epsilon>0$, since $\lambda_1(L_{\lambda})\geq 0$ we have that
$\lambda_1(L_{\lambda}-\epsilon)>0$. This implies there exists a
function $\phi_{\epsilon}$, positive in $\overline{\Omega}$, as in
Remark \ref{notestable}. Hence we have that 
$||u_{\lambda}||_{
  L^{\infty}(\Omega)}\leq C$, where $C$ is independent of $\lambda$.

\vspace{1em}

Under the same conditions as above, we can establish that the limiting
function $u^*=\lim_{\lambda\rightarrow\lambda^*} u_{\lambda}$ is a
weak solution to ($\ref{lambdau}$) with $\lambda=\lambda^*$. Just use
the weak formulation for $u_{\lambda}$ and the fact that
$u_{\lambda}\in L^{\infty}(\Omega)$, so that we can take limits in $\lambda$
and obtain that $u^*$ is a weak solution.
Therefore using the $L^{\infty}$ uniform bound on $u_{\lambda}$ we have $||u^*||_{L^{\infty}}\leq C$.
 \end{proof}

\end{section}
\begin{section}{$H^1$ regularity}

In this section we study the $H^1$ regularity of positive solutions to the equation 
\begin{equation}\label{equ}
-\Delta u -b(x)|\nabla u|^2=\lambda g(u)\quad {\rm in} \; \Omega,
\end{equation}
such that $u\equiv 0$ on $\partial\Omega$, $\Omega\subset\R^n$ is a bounded
domain, $g\geq 0$ and $g'>0$ in $\Omega$. 

We consider two cases.

\vspace{1em}

\noindent {\bf Case 1: $b(x)=b>0$ is constant}

\vspace{1em}

In this setting one can use the Hopf-Cole transformation and study
the resulting equation for the new function $v$. Then, using the
results of \cite{BV}, we have that $v$ is in $H^1$ if it is stable
and the nonlinearity $f$ satisfies
\begin{equation}\label{condfH1}
\liminf_{s\rightarrow\infty}\frac{f'(s)s}{f(s)}>1.
\end{equation}
This condition could be rewritten in terms of the nonlinearity $g(u)$ allowing
us to conclude that $e^u$, and hence $u$, is in $H^1$. It is natural
to expect that such assumptions on $g$ will be too restrictive, since they
give a condition for $e^u$, and not just $u$, to be in $H^1$. In what follows we study
directly the problem for $u$ and find the natural conditions to impose
on $g$.

\begin{prop}\label{teoH1b}
Let $b>0$ be a constant and $u$ a positive classical solution of the problem
$-\Delta u -b|\nabla u|^2 = \lambda g(u)$ with zero Dirichlet boundary
conditions, $g\geq 0$ and $g'>0$ in $\Omega$ and $\lambda>0$ a parameter. 
Assume that $u$ is a stable solution
Then, if 
$$\liminf_{s\rightarrow\infty}\frac{g'(s)(e^{bs}-1)}{bg(s)}>1,$$
we have that $||u\||_{H^1(\Omega)}\leq C$ where $C$ is independent of $\lambda$.
\end{prop}

To better understand the above condition on $g$, let us consider the case
where equality holds, i.e.,
$$\frac{g'(s)(e^{bs}-1)}{bg(s)}=1.$$
Integrating we get $\log g(s)=\log (e^{bs} -1) -bs +C$ for some constant
C and hence,
$$g(s)=C(1-e^{-bs}).$$
Recall that $b>0$ so this means that $g$ is bounded.

\vspace{1em}

As we mentioned before, this condition on $g$ is less restrictive than
the one imposed via $f$. In fact, if  $g(u)=e^u$ then $u$ is in
$H^1$ by the previous theorem. However, if we pass to the equation for $v=e^{bu}-1$ we have that
$-\Delta v=\lambda f(v)$ with  $f(v)=b(v+1)^p$, $p=1+1/b$ and $f$ does
not satisfy condition \eqref{condfH1} of \cite{BV}.

\begin{proof}

Since $u$ is stable there exists a positive function $\phi$ on
$\overline{\Omega}$ such that
$$-\Delta\phi-2b\nabla u\nabla\phi\geq \lambda g'(u)\phi.$$
Multiplying by  $(e^{bu}-1)^2/\phi$ and integrate in $\Omega$. 
\begin{eqnarray*}
 \lambda\int_{\Omega} g'(u)(e^{bu}-1)^2& \leq & \int_{\Omega}
 -\frac{|\nabla\phi|^2}{\phi^2}(e^{b u}-1)^2 + \int_{\Omega}
 2b\frac{\nabla\phi}{\phi}\nabla u e^{b u}(e^{b u}-1) -\\
& & - \int_{\Omega}
 2b\frac{\nabla\phi}{\phi}\nabla u (e^{b u}-1)^2 \\
& = & \int_{\Omega} -\frac{|\nabla\phi|^2}{\phi^2}(e^{b u}-1)^2 + 
\int_{\Omega} 2b\frac{\nabla\phi}{\phi}\nabla u (e^{b u}-1) \\
& \leq & \int_{\Omega} b^2|\nabla u|^2.
\end{eqnarray*}
On the other hand, multiplying $(\ref{equ})$ by $e^{bu}-1$ and
integrating we get
 \begin{equation}\label{demo1}
 \begin{array}{rcl}
 {\ds \lambda\int_{\Omega} g(u)(e^{bu}-1)} & = & {\ds \int_{\Omega} \nabla u\nabla(e^{bu}-1) -
 b|\nabla u|^2(e^{bu}-1)} \\
& = &{\ds  b\int_{\Omega}|\nabla u|^2.}
\end{array}
\end{equation}

Thus, we have that
\begin{equation}\label{intg}
\lambda \int_{\Omega} g'(u)(e^{bu}-1)^2 \leq \lambda b\int_{\Omega} g(u)(e^{bu}-1).
\end{equation}

Assume that
\begin{equation}\label{condg}
bg(s)(e^{bs}-1)\leq \delta g'(s)(e^{bs}-1)^2 + C
\end{equation} 
for some constant $\delta<1$ and some constant $C$.

Then, from $(\ref{intg})$ we get that 
$$\int_{\Omega} g'(u)(e^{bu}-1)^2 \leq C,$$
which implies both
$$\int_{\Omega} g(u)(e^{bu}-1) \leq C \quad \text{ and, by  }\eqref{demo1},\int_{\Omega} |\nabla
u|^2\leq C.$$

Now, going back to $(\ref{condg})$ we see that, for $s$ small it is
always possible to find $\delta$ and $C$. The problem occurs when $s$
tends to infinity (that is, when $u$ is unbounded). 
It is easy to see that \eqref{condg} holds if
\begin{equation}\label{cond}
\liminf_{s\rightarrow\infty}
\frac{g'(s)(e^{bs}-1)}{bg(s)}\geq\frac{1}{\delta}>1.
\end{equation}
\end{proof}

\vspace{2em}

\noindent {\bf Case 2: $b(x)\leq -\epsilon<0$}

\vspace{1em}

This case is, in some sense, more general than the previous one since
we do not need to assume that $u$ is a stable solution. 
The proof uses a technique due to Boccardo (see \cite{B})
involving truncations. For a function $u$ we define the truncation
$T_1u$ as
\begin{equation}\label{T1}
T_1u= \left\{\begin{array}{rl}
             1, & u>1 \\
             u, & |u|\leq 1 \\
             -1, & u<-1.
             \end{array}
\right.
\end{equation}
We have $\nabla T_1u=\nabla u$ where $|u|\leq 1$ and $\nabla T_1=0$ otherwise.

\begin{prop}\label{teoH1b-}
Let $b(x)\leq -\epsilon<0$ for some $\epsilon>0$ and $u$ a positive classical
solution to the problem $-\Delta u - b(x)|\nabla u|^2=\lambda g(u)$
with zero Dirichlet boundary conditions, $\lambda>0$ a parameter, and
assume that $g(u)\in L^1(\Omega)$.
Then, $||u||_{H^1(\Omega)}\leq C$, where $C$ is independent of $\lambda$.
\end{prop}

\begin{proof}
 We multiply equation
($\ref{equ}$) by $T_1u$ and integrate by parts
$$\int_{\Omega} \nabla u\nabla T_1u -\int_{\Omega} b(x)|\nabla u|^2 T_1u=\lambda \int_{\Omega}
g(u)T_1u.$$
Given the definition of $T_1u$  this yields
$$\int_{\{|u|\leq 1\}} |\nabla u|^2 = \lambda \int_{\Omega} g(u)T_1u + \int_{\Omega}
b(x)T_1u|\nabla u|^2.$$
Since $u$ is assumed to be positive, $b(x)\leq -\epsilon<0$ for some
$\epsilon>0$ and 
$0\leq T_1u\leq 1$ we get
$$\int_{\{u\leq 1\}} |\nabla u|^2 + \epsilon\int_{\{u>1\}}|\nabla u|^2\leq \lambda \int_{\Omega} |g(u)| .$$\end{proof}
\end{section}

%\end{section}

\end{document}